\documentclass[12pt]{article}

\usepackage{color}

\usepackage{amssymb,latexsym,amsmath, amsfonts}
\usepackage{amsthm,indentfirst}

\newcommand{\emm}{\mathrm{1\hspace{-0.3em} I}}

\newtheorem{thm}{Theorem}
\newtheorem{dfn}[thm]{Definition}

\newtheorem{cor}[thm]{Corollary}

\def\B{{\mathfrak{B}}}
\def\ext{{\rm{ext} \,}}
\def\N{{\mathbb{N}}}

\begin{document} 

\begin{center}
{\Large Dilation invariant Banach limits}\\
E. Semenov, F. Sukochev, A. Usachev~\footnote{Corresponding author}, D. Zanin
\end{center}

\begin{abstract}
We study two subclasses of Banach limits: the one consisting of Banach limits which are invariant with respect of the Ces\`aro operator and another one consists of Banach limits which are invariant with respect to all dilations. We prove that the first is a proper subset of the second. We also show that these classes are at the maximal distance from the set $\ext \B$ of all extreme points of the set of all Banach limits.
\end{abstract}

\begin{flushright}
To the memory of W. A. J. Luxemburg.
\end{flushright}

\section{Introduction and Preliminaries}

Throughout the paper we denote by $\ell_\infty$ the space of all bounded sequences $x=(x_1,x_2,\ldots)$ equipped with the norm
\[
 \|x\|_{l_\infty}=\sup\limits_{n\in\N}|x_n|,
\]
and the usual partial order. Here $\N$ stands for the set of natural numbers.

\begin{dfn}
 A linear functional $B\in \ell_\infty^*$ is said to be a Banach limit  if
\begin{enumerate}
\item
$B\geqslant0$, that is $Bx\geqslant0$ for every $x\geqslant0$,

\item
$B\emm=1$, where $\emm=(1,1,\ldots)$,

\item
$B(Tx)=B(x)$ for every $x\in \ell_\infty$, where $T$ is a translation operator,
that is \\ $T(x_1,x_2,\ldots)=(0,x_1,x_2,x_3,\ldots)$.
\end{enumerate}
\end{dfn}

The existence of Banach limits was established by S.~Mazur and then appeared in the book of S.~Banach~\cite{B}.
 Since then Banach limits were proved to be a useful tool in functional analysis. 
Recent developments in the theory of operator ideals and singular traces~\cite{DPSS, DPSSS1, DPSSS2, SSUZ} reformulated 
important and difficult problems in these areas in terms of Banach limits.

We denote the set of all Banach limits by $\B$. It is easy to see that $\B$
is a closed convex set on the unit sphere of $\ell_\infty^*$. Hence, $\|B_1-B_2\|_{\ell_\infty^*}\leqslant2$ for every
$B_1,B_2\in\B$. It follows directly from the definition of Banach limits that
\[
 \liminf\limits_{n\to\infty} x_n\leqslant
 Bx\leqslant
 \limsup\limits_{n\to\infty} x_n
\]
for every $x\in \ell_\infty$, $B\in \B$. In particular,
$Bx=\lim\limits_{n\to\infty} x_n$ for every convergent sequence $x\in \ell_\infty$. 

Following G. G. Lorentz~\cite{L}, we call a sequence $x\in \ell_\infty$
almost convergent to a number $a\in \mathbb R$, if $B(x)=a$ for every $B\in\B$.

We denote the set of all almost convergent sequences (resp., sequences almost
conver\-gent to zero) by $ac$ (resp., $ac_0$). The following criterion for almost convergence
was proved by Lorentz~\cite{L}.

\begin{thm}\label{lorentz}
 A sequence $x$ is almost convergent to $a$  if and only if
\[
 \lim\limits_{n\to\infty}\frac1n\sum\limits_{k=m+1}^{m+n}x_k=a
\]
uniformly in  $m\in\N$.
\end{thm}

It is easy to see that Banach limits are not multiplicative functionals on $\ell_\infty$. Indeed, for the sequence $x=(1,0,1,0, ...)$ we have $x+Tx=\emm$ and $B(x)=1/2$. Hence,
$$0=B(x\cdot Tx)\neq B(x) \cdot B(Tx)=\frac14.$$
This fact sets apart the set of all Banach limits and the set of all extreme points of the set of all extended limits (or, states on the algebra $\ell_\infty$), which coincides with the set of all bounded multiplicative functionals (or, pure states) on this algebra. 
Moreover, the distance between Banach limits and the bounded multiplicative functionals is 2. Indeed, for every multiplicative functional $\gamma$ on $\ell_\infty$ and every $x\in \{0,1\}^\N$ (the set of all sequences of zeroes and ones) we have $\gamma(x)=\gamma(x^2)=(\gamma(x))^2$. Thus, $\gamma(x)$ is either $0$ or $1$. Fix $j\in\N$ and take $x=\chi_{j\N}\in \{0,1\}^\N$. Since $\sum_{i=0}^{j-1} T^ix=\emm$, if follows that $\gamma(T^ix)=1$ for some $0\le i\le j-1$. On the other hand, the sequence $x$ is almost convergent and $Bx=1/j$ for every Banach limit $B$. Thus,
$$\|\gamma-B\|_{\ell_\infty^*}\ge (\gamma-B)(2T^ix-\emm)=2(1-\frac1j).$$ 
Taking $j\to\infty$, proves the claim.

However, some Banach limits are multiplicative on some subspace of $\ell_\infty$. More precisely, in \cite{Luxemburg} W. A. J. Luxemburg proved that every extreme point of the set $\B$ is multiplicative on the stabiliser of $ac_0$ defined as follows:
$${\rm St_{ac_0}}:=\{z\in \ell_\infty : z\cdot ac_0 \subset ac_0\}.$$
This set was further investigated in \cite{Alekhno, SSU2, ASSU2}. An analogue of this set for Banach limits with additional invariance properties was introduced in \cite{ASSU4}. The paper \cite{Luxemburg} is also interesting for its approach to study Banach limits via non-standard analysis. It is also worth to mention the work \cite{LP} of W. A. J. Luxemburg and B. de Pagter, who introduced the notion of almost convergent elements with respect to an abelian semigroup of Markov operators on AM-spaces.

The Lorentz' result was strengthened by L. Sucheston~\cite{S} as follows:
\begin{thm}\label{sucheston}
 For every $x\in \ell_\infty$ one has
\[
 \{Bx:B\in\B\}=[q(x),p(x)],
\]
where
\[
 q(x)=\lim\limits_{n\to\infty}
 \inf\limits_{m\in\N} \frac1n\sum\limits_{k=m+1}^{m+n}x_k,
 \quad
 p(x)
 =\lim\limits_{n\to\infty}
 \sup\limits_{m\in\N} \frac1n\sum\limits_{k=m+1}^{m+n}x_k.
\]
\end{thm}

The first result concerning Banach limits with additional invariance properties is due to W.~Eberlein~\cite{Eberlein}, who established the existence
of Banach limits invariant under the Hausdorff transformations.
Eberlein's approach was further developed in~\cite{SS_JFA} to study Banach limits invariant under an operator $H \in \Gamma$. Here, $\Gamma$ stands for the set of all linear operators in $\ell_\infty$,
which satisfy the following conditions:
\begin{enumerate}
\renewcommand{\labelenumi}{(\roman{enumi})}
\item
$H\geq0$ and $H\emm=\emm$;

\item
$Hc_0\subset c_0$;

\item
$\limsup\limits_{j\to\infty}(A(I-T)x)_j\geq0$ for all $x\in
l_\infty$, $A\in
R$, where $R=R(H)={\rm conv}\{H^n,n\in\mathbb N\cup\{0\}\}$.
\end{enumerate}


Denote by $ \mathfrak B(H)$ the set of all Banach limits invariant with respect to the operator $H$.

For every $m \in \mathbb{N}$ define a dilation operator $\sigma_m : \ell_\infty \to \ell_\infty$ as follows:
\begin{equation}\label{26}
\sigma_m(x_1,x_2,\ldots)=(\underbrace{x_1,x_1,\ldots,x_1}_m, \underbrace{x_2,x_2,\ldots,x_2}_m, \ldots).
\end{equation}

Denote by $C$ the Ces\`aro operator on $\ell_\infty$ given by
$$(Cx)_n=\frac1n \sum_{k=1}^n x_k, \ n\ge1.$$

It can be shown that $C, \sigma_m\in \Gamma$ for every $m\in\N$ (see discussion after Theorem 8 in \cite{ASSU2} and \cite{SS_JFA}) and, so the sets $\mathfrak B(\sigma_m)$ and $\mathfrak B(C)$ are non-empty.
 The sets $\mathfrak B(\sigma_m)$ were studied in \cite{DPSSS2, ASSU2, SSUZ2}. In the present paper we investigate the set $\cap_{m\ge 2} \mathcal{B}(\sigma_m)$.
 We also discuss properties of the set $\ext \B$ of all extreme points of the convex set $\mathfrak{B}$.

\section{Relation to Ces\`aro invariant Banach limits}

It was shown in \cite[Theorem 3]{SSUZ2} that for every $m\in \N$ the following inclusion holds $\mathfrak B(C) \subset \mathfrak B(\sigma_m)$.  Hence,
\begin{equation}\label{eq_main}
\mathfrak B(C) \subset  \bigcap_{m=2}^\infty\mathfrak B(\sigma_m).
\end{equation}

In \cite[Theorem 4.8]{ASSU4} it was shown that the inclusion $\mathfrak B(C) \subset \mathfrak B(\sigma_m)$ is also proper.
In this section we show that the inclusion \eqref{eq_main} is proper.

Set $\Sigma:=\bigcap_{m=2}^\infty\mathfrak B(\sigma_m).$
 For a given $x\in \ell_\infty$ denote
$$\mathfrak B(C, x)=\{B(x) : B\in \mathfrak B(C)\} \ \text{and} \ \mathfrak B(\Sigma, x)=\{B(x) : B\in \Sigma\}.$$


\begin{thm}
There exists $B\in \Sigma$ such that $B\notin \mathfrak B(C).$
\end{thm}
  
\begin{proof}
It was proved in~\cite[Theorem 7]{ASSU2} that:
\begin{equation}\label{c2}
\mathfrak B(\Sigma, x)=\left[\sup_{A\in R(\Sigma), s\in S} \liminf_{j\to\infty} (A(x+s))_j, \inf_{A\in R(\Sigma), s\in S} \limsup_{j\to\infty} (A(x+s))_j  \right],
\end{equation}
where $S:=(I-T)\ell_\infty$ and $R(\Sigma)$ is a convex hull of all finite products of operators $\sigma_m$. Note that for every $n\in\N$ and $k_1, ..., k_n \in \N$ we have $\sigma_{k_1}\sigma_{k_2}\cdots \sigma_{k_n}=\sigma_{k_1 k_2\cdots k_n}$. Thus, $R(\Sigma)$ is nothing but $ {\rm conv}\{\sigma_n, n\in\N\}$.

It is easy to see, that for every $B\in \mathfrak B(C)$ and every $x\in \ell_\infty$ such that $\lim_{j\to\infty} (Cx)_j=0$ one has $Bx=B(Cx)=0.$

Taking these facts into account, in order to prove the assertion it is sufficient to construct $x\in \ell_\infty$ such that
\begin{equation}\label{c4}
 \lim_{j\to\infty} (Cx)_j =0
\end{equation}
and
$$\limsup_{j\to\infty} (A(x+s))_j =1 $$
for every $A\in {\rm conv}\{\sigma_n, n\in\N\}$ and $s\in S$.

For a given $n\in \N$ denote $J_n = [2^{2^n}-n, 2^{2^n})$ and let $J_{n,k}$ be a minimal by inclusion subset of $\N$ such that $\sigma_k J_{n,k} \supset J_n$ for all $1\le k\le n$. Set $I_n = \cup_{k=1}^n J_{n,k}$. Since
$$|J_{n,k}|\le \frac1k |J_n|+2=\frac{n}{k}+2,$$
it follows that
$$|I_n|\le \sum_{k=1}^n (\frac{n}{k}+2)= n \sum_{k=1}^n \frac{1}{k}+2n<n(n+2)$$
and $(2^{2^{n-1}}, 2^{2^n}] \supset I_n$ for all $n\in \N$. Consider the sequence
$$x_m=
\begin{cases}
1, \  m\in \bigcup_{k=1}^\infty I_{k}\\
0, \ \text{otherwise}.
\end{cases}$$

If $j\in (2^{2^{n-1}}, 2^{2^n}]$ for some $n\in \N$, then 
$$(Cx)_j \le 2^{-2^{n-1}}\sum_{k=1}^n |I_k| < 2^{-2^{n-1}}\sum_{k=1}^n k(k+2)$$
and, thus,
$$\lim_{j\to\infty} (Cx)_j=0.$$

Let now $A\in {\rm conv}\{\sigma_n, 1\le n\le m\}$ for some $m\in \N$ and $s=(I-T)y$ for some $y\in \ell_\infty$. We have
$$\left| \sum_{i=k+1}^{k+r} s_i \right| = |y_{k+1}-y_{k+r}|\le 2 \|y\|_{\ell_\infty}$$
for every $k, r \in \N$. Hence,
$$\left| \sum_{i=k+1}^{k+r} (\sigma_n s)_i \right| \le 2 n \|y\|_{\ell_\infty}\le 2 m \|y\|_{\ell_\infty}$$
for every $n=1,2,\dots, m$. Since $A\in {\rm conv}\{\sigma_n, 1\le n\le m\}$, it follows that
$$\left| \sum_{i=k+1}^{k+r} (A s)_i \right| \le \max_{1\le n\le m}\left| \sum_{i=k+1}^{k+r} (\sigma_n s)_i \right| \le 2 m \|y\|_{\ell_\infty}$$
Therefore,
$$\max_{k<i\le k+r} (As)_i  \ge - \frac{2 m \|y\|_{\ell_\infty}}{r}$$
for every $k, r \in \N$. Since the above inequality holds for arbitrary large $r$, it follows that
\begin{equation}\label{c5}
\max_{k<i\le k+r} (As)_i  \ge 0.
\end{equation}

By construction $\sigma_kx\vert_{J_n}=1$ for all $1\le n\le m$. Hence, $Ax\vert_{J_n}=1$. Next if $n\ge 4 m \|y\|_{\ell_\infty},$ then by \eqref{c5}
$$\max_{j\in J_n} (A(x+s))_j  \ge 1- 0=1. $$
This means that 
$$\inf_{A\in R(\Sigma), s\in S} \limsup_{j\to\infty} (A(x+s))_j \ge 1. $$ The converse inequality is straightforward. 
This proves the assertion.
\end{proof}

\section{Distance to extreme points}

Denote by $\ext\B$ the set of extreme points of $\B$. 
Since $\B$ is compact in the weak$^*$ topology, it follows from Krein-Milman theorem that
\[
 \B=\overline{\rm conv} \ \ext\B,
\]
where the closure is taken in the weak$^*$ topology. 
The classical and recent studies of the set of extreme points of $\B$ demonstrate that this set has very interesting properties~\cite{Alekhno, Luxemburg, Nillsen, SS_pos, SSU2, Talagrand}.
In particular, it was shown in~\cite{SS_pos} that a subspace generated by any countable sequence in $\ext \B$ is isometric to
$l_1$ . Hence, $\|B_1-B_2\|=2$ for every $B_1,B_2\in
ext\,\B$, $B_1\not=B_2$.

It was shown in \cite[Theorem 14]{ASSU2} that the sets $\B(\sigma_n)$, $n\in \N$, $n\ge2$ and $\ext \B$ are on the maximal distance from each other.
%
%
%
%
This fact can be given in the stronger form (see Corollary \ref{cor6} below). We need some preparations to state the result. The kernel $N(I - T^*)$ of an operator $I - T^*$ is a Riesz subspace in the Banach lattice $\ell_\infty^*$ (see e.g.,~\cite{Alekhno}). So, it is an $AL$-subspace under the norm and the order induced from $\ell_\infty^*$. Clearly, the positive part of the unit sphere $S^+_{N(I - T^*)}$ coincide with the set $\B$. 

As an $AL$-space $N(I - T^*)$ is isometrically order isomorphic to the space $L_1(\Omega)$ of all functions integrable on some set $\Omega$ with a measure~$\mu$.
The set
$\Omega$ can be partitioned into $\Omega_d$ and $\Omega_c$ such that the restriction of $\mu$ to $\Omega_d$
is discrete and the restriction of $\mu$ to $\Omega_c$ is continuous.
That is,
\begin{equation*}
N(I - T^*) \approx L_1(\Omega) =
L_1(\Omega_d)\oplus L_1(\Omega_c).
\end{equation*}
 Atoms of
$\Omega_d$ correspond to ${\rm ext\! \
}\mathfrak{B}$. So, $L_1(\Omega_d)$ is isometrically isomorphic to the closure of a hull of ${\rm ext\! \ }\mathfrak{B}$ in the norm topology.
In particular, the closure of 
${\rm conv}({\rm ext\! \ }\mathfrak{B})$ in the norm topology coincide with the positive part $S^+_{L_1(\Omega_d)}$ of a unit sphere of the space
$L_1(\Omega_d)$. 
We refer to \cite{ASSU2} for more details.

Denote $\B_d := S^+_{L_1(\Omega_d)}$ and $\B_c := S^+_{L_1(\Omega_c)}$. Since components $L_1(\Omega_d)$ and $L_1(\Omega_c)$ are disjoint in $N(I - T^*)$, it follows that $\|B_1-B_2\|_{\ell_\infty^*}=2$ for every $B_1\in\B_d$ and $B_2\in\B_c$.

\begin{thm}\label{T5}
For every $n\in \N$, $n\ge2$ one has  $\B(\sigma_n)\subset \B_c.$
\end{thm}

\begin{proof}
For $B_1 \in \B(\sigma_n)$ and $B_2 \in \ext \B$ consider the function 
$$g(t):=\|B_1-tB_2\|_{\ell_\infty^*}, \ t\ge 0.$$
It is easy to see that $g$ is a convex function and $g(t)\le 1+t$ for all $t\ge0$. From \cite[Theorem 14]{ASSU2} we obtain that $g(1)=2$. These three facts imply that $g(t)= 1+t$ for all $t\ge0$, that is
\begin{equation}\label{eq1}
\|B_1-tB_2\|_{\ell_\infty^*} = 1+t\ge 1, \ t\ge 0.
\end{equation}

Suppose to the contrary, that $\B(\sigma_n)$ is not a subset of $\B_c.$ Hence, there exist $B_3, B_4 \in \B(\sigma_n)$, $B_i\in \ext \B$, $5\le i\le M\le \infty$ and numbers $\alpha_4\ge0$, $\alpha_i>0$ for all $5\le i\le M\le \infty$ such that 
$$B_3=\sum_{i=4}^M \alpha_i B_i \quad \text{and} \quad \sum_{i=4}^M \alpha_i=1.$$

We have 
$$\|B_3-\alpha_5B_5\|_{\ell_\infty^*}= \|\sum_{\substack{i=4 \\ i\neq 5}}^M \alpha_i B_i\|_{\ell_\infty^*}= \sum_{\substack{i=4 \\ i\neq 5}}^M \alpha_i <1,$$
which contradicts \eqref{eq1}. Therefore, $\B(\sigma_n)\subset \B_c.$
\end{proof}

\begin{cor}\label{cor6}
For every $n\in \N$, $n\ge2$ and every $B_1\in\B(\sigma_n)$, $B_2\in\overline{\rm conv}^n\ext \B$ (the closure is taken in the norm topology) one has $\|B_1-B_2\|_{\ell_\infty}=2.$
\end{cor}

\begin{proof}
By Theorem \ref{T5}, $\B(\sigma_n)\subset \B_c$. The assertion follows from the discussion above Theorem \ref{T5}.
\end{proof}

%
%
%
%

\section{The special subclass of Banach limits}

Let $\gamma:\ell_{\infty}\to\mathbb{C}$ be an extended limit, that is a positive linear functional such that $\gamma(x)=\lim_{n\to\infty} x_n$ for every convergent $x\in\ell_\infty$. Consider the functional
$$Bx=\log(2)\cdot\gamma\Big(\frac1n\sum_{k\geq0}x_k 2^{-k/n}\Big), \ x\in \ell_\infty.$$ 
It was proved in \cite[Lemma 3.5]{SUZ4} that for every extended limit $\gamma$ the functional $B$ is a Banach limit. Denote the set of all such Banach limits by $\mathfrak B_\zeta$.
These are exactly the Banach limits which correspond (in the sense of A. Pietsch \cite{SSUZ}) to the extended $\zeta$-function residues (see \cite{SUZ4} and references therein).

\begin{thm}\label{B_ext}
The set $\mathfrak B_\zeta$ is disjoint with the set $\ext \B$.
\end{thm}

\begin{proof}
It was shown in the proof of Lemma 3.7 in \cite{SUZ4} that every $B\in \mathfrak B_\zeta$ can be written in the form $B = B_1 \circ C$ for some Banach limit $B_1$. Since every Banach limit  is an extended limit, it follows that
$$\liminf_{k\to\infty} x_k \le B_1x \le \limsup_{k\to\infty} x_k$$
for every $x\in \ell_\infty$. Thus,
$$\liminf_{k\to\infty} (Cx)_k \le Bx \le \limsup_{k\to\infty} (Cx)_k$$
for every $x\in \ell_\infty$.

Consider the sequence 
$$x = \sum_{k=0}^\infty \chi_{[4^k, 2\cdot 4^k)}\in \ell_\infty.$$

It follows from~\cite[Proposition 2.4]{KM} that for every $B\in {\rm ext} \ \mathfrak B$ the value of $Bx$ is either $0$ or $1$. 
Thus, it is sufficient to show that that $0< Bx<1$.

Direct computations yield
$$ Bx \le \limsup_{n\to\infty} (Cx)_n = \limsup_{n\to\infty} \frac1{2\cdot 4^n}\sum_{k=0}^{2\cdot 4^n}x_k=\limsup_{n\to\infty} \frac1{2\cdot 4^n}\sum_{k=0}^{n}4^k=\frac23$$
and
$$ Bx \ge \liminf_{n\to\infty} (Cx)_n = \liminf_{n\to\infty} \frac1{4^n-1}\sum_{k=0}^{4^n-1}x_k=\liminf_{n\to\infty} \frac1{4^n-1}\sum_{k=0}^{n-1}4^k=\frac13.$$

This completes the proof.
\end{proof}

Acknowledgment. The work of the first and the third authors was supported by the Russian Science Foundation (grant 19-11-00197), the work of the second and the fourth authors was supported by the Australian Research Council.

\vspace{10mm}

Evgenii Semenov, Faculty of Mathematics, Voronezh State University, Voronezh, 394006 Russia. \texttt{semenov@math.vsu.ru}

Fedor Sukochev, School of Mathematics and Statistics, The University of New South Wales, 2052, NSW, Australia. \texttt{f.sukochev@unsw.edu.au}

Alexandr Usachev, School of Mathematics and Statistics, Central South University, Hunan, China. \texttt{alex.usachev.ru@gmail.com}

Dmitriy Zanin, School of Mathematics and Statistics, The University of New South Wales, 2052, NSW, Australia. \texttt{d.zanin@unsw.edu.au}

\begin{thebibliography}{10}

\bibitem{Alekhno}
{ E.A. Alekhno}, {\it Superposition operator on the space of sequences almost  converging to zero}, Cent. Eur. J. Math. {10}
 (2012), 619--645.

\bibitem{ASSU2}
{ E. Alekhno, E. Semenov, F. Sukochev, A. Usachev},
\newblock  \textit{Order and geometric properties of the set of Banach
limits},
\newblock { St. Petersburg Math. J.}  (2016),  3--35.

\bibitem{ASSU4} 
E. Alekhno, E. Semenov, F. Sukochev, A. Usachev \textit{Invariant Banach limits and their extreme points}, Studia Math., \textbf{242}:1 (2018), 79--107.


\bibitem{B}
{\sc Banach, S.}
\newblock {\em Th\'eorie des op\'erations lin\'eaires}.
\newblock \'Editions Jacques Gabay, Sceaux, 1993.
\newblock Reprint of the 1932 original.
%
%
\bibitem{DPSSS1}
{\sc Dodds, P.~G., de~Pagter, B., Sedaev, A.~A., Semenov, E.~M., and Sukochev,
  F.~A.}
\newblock Singular symmetric functionals.
\newblock {\em Zap. Nauchn. Sem. S.-Peterburg. Otdel. Mat. Inst. Steklov.
  (POMI) 290}, Issled. po Linein. Oper. i Teor. Funkts. 30 (2002), 42--71, 178.

\bibitem{DPSSS2}
{\sc Dodds, P.~G., de~Pagter, B., Sedaev, A.~A., Semenov, E.~M., and Sukochev,
  F.~A.}
\newblock Singular symmetric functionals and {B}anach limits with additional
  invariance properties.
\newblock {\em Izv. Ross. Akad. Nauk Ser. Mat. 67}, 6 (2003), 111--136.

\bibitem{DPSS}
{\sc Dodds, P.~G., de~Pagter, B., Semenov, E.~M., and Sukochev, F.~A.}
\newblock Symmetric functionals and singular traces.
\newblock {\em Positivity 2}, 1 (1998), 47--75.

\bibitem{Eberlein}
{\sc Eberlein, W.~F.}
\newblock Banach-{H}ausdorff limits.
\newblock {\em Proc. Amer. Math. Soc. 1\/} (1950), 662--665.
%
\bibitem{KM}
{\sc Keller, G., and Moore, Jr., L.~C.}
\newblock Invariant means on the group of integers.
\newblock In {\em Analysis and geometry}. Bibliographisches Inst., Mannheim,
  1992, pp.~1--18.
  
  \bibitem{L}
{\sc Lorentz, G.~G.}
\newblock A contribution to the theory of divergent sequences.
\newblock {\em Acta Math. 80\/} (1948), 167--190.




\bibitem{Luxemburg}
{ W. A. J. Luxemburg}, Nonstandard hulls, generalized limits and almost convergence, \textit{Analysis and Geometry, Bibliographisches Inst., Mannheim}, (1992), 19--45.

\bibitem{LP}
W. A. J. Luxemburg, B. de Pagter,  {\it Invariant Means for Positive Operators and Semigroups},
Katholieke Universiteit Nijmegen, Subfaculteit Wiskunde, Nijmegen, (2005) 31--55.

\bibitem{Nillsen}
{\sc Nillsen, R.}
\newblock Nets of extreme {B}anach limits.
\newblock {\em Proc. Amer. Math. Soc. 55}, 2 (1976), 347--352.

\bibitem{SS_pos}
{\sc Semenov, E.~M., and Sukochev, F.~A.}
\newblock Extreme points of the set of {B}anach limits.
\newblock {\em Positivity 17}, 1 (2013), 163--170.

\bibitem{SS_JFA}
{ E.M. Semenov, F.A. Sukochev}, \textit{Invariant Banach limits and applications}, J. Funct. Anal. {256} (2010), 1517--1541.

\bibitem{SSU2}
{ E.~M. Semenov, F.~A. Sukochev, A.~S. Usachev},
\textit{ Geometric properties of the set of {B}anach limits},
\newblock { Izv. Ross. Akad. Nauk Ser. Mat. } {78} (2014), 177--204.
%
\bibitem{SSUZ}
{\sc Semenov, E.~M., Sukochev, F.~A., Usachev, A.~S., and Zanin, D.~V.}
\newblock Banach limits and traces on $\mathcal{L}_{1,\infty}$.
\newblock {\em Adv. Math. 285\/} (2015), 568--628.



\bibitem{SSUZ2}
 E. Semenov, F. Sukochev, A. Usachev, D. Zanin \textit{Invariant Banach limits and applications to noncommutative geometry}, submitted manuscript.
 
 \bibitem{S}
{\sc Sucheston, L.}
\newblock Banach limits.
\newblock {\em Amer. Math. Monthly 74\/} (1967), 308--311.

\bibitem{SUZ4}
{\sc Sukochev, F., Usachev, A., and Zanin, D.}
\newblock Singular traces and residues of the $\zeta$-function.
\newblock {\em Indiana J. Math. 66}, 4 (2017), 1107--1144.

\bibitem{Talagrand}
{\sc Talagrand, M.}
\newblock Moyennes de {B}anach extr\'emales.
\newblock {\em C. R. Acad. Sci. Paris S\'er. A-B 282}, 23 (1976), Aii,
  A1359--A1362.

\end{thebibliography}
\end{document}